\documentclass[a4paper, 12pt]{article}
\usepackage[left=3cm, right=2cm, top=2.5cm, bottom=2.5cm]{geometry}

\usepackage{amssymb,amsmath,amsthm,amsfonts,mathrsfs,bbm}
\usepackage[pdftex]{graphicx}
\usepackage{subcaption}
\usepackage{hyperref}
\usepackage{cleveref}
\usepackage{listings}
\usepackage[T1]{fontenc}
\usepackage[utf8]{inputenc}
\usepackage[english]{babel}

\usepackage{algpseudocode}

\usepackage{paralist,enumitem,tabularx}
\usepackage{mathtools}

\usepackage{layout}
\usepackage{fancyhdr}
\usepackage{setspace}

\usepackage{float}

\usepackage{pdfpages}

\usepackage{datetime}

\usepackage[backend=bibtex,maxnames=99]{biblatex} 
\numberwithin{equation}{section}

\theoremstyle{bolddef}
\newtheorem{definition}{Definition}[section]
\newtheorem{algorithm}[definition]{Algorithm}
\newtheorem{assumption}[definition]{Assumption}
\newtheorem{example}[definition]{Example}

\theoremstyle{boldplain}
\newtheorem{lemma}[definition]{Lemma}
\newtheorem{theorem}[definition]{Theorem}

\newcommand{\xk}{x^k}
\newcommand{\xkm}{x^{k-1}}
\newcommand{\xkp}{x^{k+1}}

\newcommand{\xhkp}{{\hat x}^{k+1}}
\newcommand{\pmin}{p_{\min}}
\newcommand{\gammag}{\underline{\gamma}_\rho}

\newcommand{\R}{\mathbb{R}}
\newcommand{\barR}{\overline{\mathbb{R}}}
\newcommand{\N}{\mathbb{N}}
\newcommand{\X}{\mathbb{X}}
\newcommand{\dom}{\mathrm{dom}}
\newcommand{\Rvalue}{\mathcal{R}}
\newcommand{\dist}{\mathrm{dist}}

\date{November 18, 2024}

\title{Convergence of Nonmonotone Proximal Gradient Methods under the Kurdyka–Łojasiewicz Property without a Global Lipschitz Assumption}

\author{Christian Kanzow\footnote{University of W\"urzburg, Institute
	of Mathematics, Emil-Fischer-Str.\ 30, 97074 W\"urzburg, Germany; 
	e-mail: christian.kanzow@uni-wuerzburg.de}
	\and Leo Lehmann\footnote{University of W\"urzburg, Institute
		of Mathematics, Emil-Fischer-Str.\ 30, 97074 W\"urzburg, Germany; 
		e-mail: leo.lehmann@stud-mail.uni-wuerzburg.de} }

\begin{document}

\maketitle

\begin{abstract}
We consider the composite minimization problem with the objective 
function being the sum of a continuously differentiable and a merely
lower semicontinuous and extended-valued function. The proximal 
gradient method is probably the most popular solver for this class
of problems. Its convergence theory typically requires that either
the gradient of the smooth part of the objective function is globally
Lipschitz continuous or the (implicit or explicit) a priori assumption 
that the iterates generated by this method are bounded. Some recent 
results show that, without these assumptions, the proximal gradient 
method, combined with a monotone stepsize strategy, is still globally convergent with a suitable rate-of-convergence under the 
Kurdyka–Łojasiewicz property. For a nonmonotone stepsize strategy,
there exist some attempts to verify similar convergence results, but,
so far, they need stronger assumptions. This paper is the first which
shows that nonmonotone proximal gradient methods for composite 
optimization problems share essentially the same nice global and 
rate-of-convergence properties as its monotone counterparts, still 
without assuming a global Lipschitz assumption and without an 
a priori knowledge of the boundedness of the iterates.
\end{abstract}

\noindent
\small\textbf{Keywords.}
Composite Optimization, Nonsmooth Optimization, Proximal Gradient Method,
Kurdyka–Łojasiewicz Property, Nonmonotone Line Search, Global Convergence,
Linear Convergence
\par\addvspace{\baselineskip}

\noindent
\small\textbf{AMS subject classifications.}
49J52, 90C30
\par\addvspace{\baselineskip}

\section{Introduction}\label{Sec:Intro}

We consider the composite optimization problem 
\begin{equation}\label{Eq:ComOpt}
   \min_x \psi(x) := f(x) + \phi(x), \quad x \in \mathbb{X},
\end{equation}
where $f:\X \to \R $ is continuously differentiable, 
$\phi: \X \to \barR $ 
is proper and lower semicontinuous, and $ \X $ denotes a 
Euclidean space (finite-dimensional Hilbert space). 
Note that neither $ f $ nor $ \phi $ has to be convex.
Composite optimization problems of this kind occur frequently in many
applications like machine learning, data compression, matrix completion, 
image processing, low-rank approximation, 
or dictionary learning \cite{BrucksteinDonohoElad2009,Chartrand2007,DiLorenzoLiuzziRinaldiSchoenSciandrone2012,DumitrescuIrofti2018,LiuDaiMa2015,MarjanovicSolo2012,Markovsky2018}. 

In many applications, the nonsmooth term $ \phi $ plays the role of 
a sparsity, regularization, or penalty term. For example, the standard 
choice $ \phi (x) := \lambda \| x \|_1 $ for $ x \in \R^n $ and
a constant $ \lambda > 0 $ is known to impose some sparsity in 
different applications. On the other, improved sparsity can often be obtained by using terms like $ \phi (x) := \lambda \| x \|_p $ with
$ p \in (0,1) $ or $ \phi (x) := \lambda \| x \|_0 $, with $ \| x \|_0 $
denoting the number of nonzero components of the vector $ x $. Note that
these two latter functions $ \phi $ are nonconvex, the second one is
even discontinuous, but lower semicontinuous. The same is true
if $ \phi $ represents the indicator function of a nonempty and
closed feasible set. Furthermore, we stress that the standard
global Lipschitz assumption on the derivative of $ f $ is satisfied
for quadratic objective functions, but usually violated for more
general nonlinear functions $ f $. In particular, this non-Lipschitz
behaviour occurs if $ f $ stands for an augmented Lagrangian function
or for the dual of a (not necessarily uniformly convex) primal problem,
cf.\ \cite{DeMarchiJiaKanzowMehlitz2022,XiaoxiKanzowMehlitz2023}. 
This indicates that it
is reasonable to consider composite optimization problems in this 
general setting.

The standard solver for composite optimization problems is the
proximal gradient method given by the iteration
\begin{equation}\label{Eq:k-subproblem}
	\xkp := \text{argmin}_{x \in \mathbb{X}} 
	\Big\{ f(\xk) + \big\langle f'(\xk), x-\xk \big\rangle
	+ \frac{1}{2 \gamma_k} \| x - \xk \|^2 + \phi (x) \Big\}
\end{equation}
for some given $ \gamma_k > 0 $, i.e., the new iterate $ \xkp $ is 
obtained by using a simple quadratic approximation of the nonlinear
and smooth function $ f $, whereas the nonsmooth function $ \phi $
is moved into the subproblem without any modifications. Using some simple
algebraic manipulations, it is well-known and easy to see that this
procedure can be rewritten as
\begin{equation}\label{Eq:ProxIteration}
	\xkp := \textrm{Prox}_{\gamma_k \phi} \big( x^k - \gamma_k f'(\xk)
	\big)
\end{equation}
with the prox operator defined by
\begin{equation*}
	\textrm{Prox}_{\gamma \phi}(x) := \text{argmin}_{z \in \mathbb{X}}
	\Big\{ \phi (z) + \frac{1}{2 \gamma} \| z - x \|^2 \Big\} .
\end{equation*}
Observe that \eqref{Eq:k-subproblem} reduces to the steepest descent
method $ \xkp := \xk - \gamma_k f'(\xk) $ for $ \phi \equiv 0 $,
which indicates
that the parameter $ \gamma_k > 0 $ may be viewed as a stepsize.

Note that the proximal gradient method allows an efficient implementation
only if the prox operator can be evaluated in a simple way. Fortunately,
there exist several practically relevant scenarios where this prox
operator can be computed either analytically or with a very small
computational overhead, cf.\ the excellent monograph \cite{Beck2017}
by Beck for a list of examples. The same monograph is also 
a perfect reference for the existing convergence theory of 
proximal point methods, at least for the case of convex functions
$ \phi $.

Extensions of the convergence theory to nonconvex and only lower 
semicontinuous functions $ \phi $ occur in the seminal papers
\cite{AttouchBolteSvaiter2013,BolteSabachTeboulle2014}. Both
papers cover global and rate-of-convergence results, and their technique
is based on a global Lipschitz assumption for the gradient of $ f $
and the \emph{Kurdyka–Łojasiewicz (KL) property}. While the latter
seems indispensable, the former assumption is very strong especially
for applications with non-quadratic functions $ f $. In some other
papers, this kind of global Lipschitz assumption is hidden by the a priori
condition that the iterates remain bounded or the assumption that
the level sets of $ \psi $ are bounded, which itself implies the
boundedness of the iterates.

In the recent paper \cite{KanzowMehlitz2022}, it was shown that
global convergence can also be obtained without this global 
Lipschitz assumption, only the much weaker condition of $ f' $
begin locally Lipschitz continuous is required (depending on the
smoothness properties of $ \phi $, it might also be enough to have
$ f $ only continuously differentiable). Moreover, the subsequent
work \cite{XiaoxiKanzowMehlitz2023} also shows that one can
obtain the standard global and 
rate-of-convergence results under the KL property,
again without the assumption that the derivative 
of $ f $ is globally Lipschitz
continuous. These results were obtained for a proximal gradient method
with a suitably updated stepsize parameter $ \gamma_k $ in such a way
that the sequence of function values $ \{ \psi (\xk) \} $ is 
monotonically decreasing.

The situation is more delicate when we allow a nonmonotone stepsize
rule, i.e., a choice of $ \gamma_k $ such that the sequence
$ \{ \psi (\xk) \} $ is no longer monotonically decreasing. The two
most prominent monotonicity strategies are the max-type rule by
Grippo et al.\ \cite{GrippoLamparielloLucidi1986} as well as the 
mean-type rule \cite{ZhangHager2004}
by Zhang and Hager. Both were originally introduced for smooth
unconstrained optimization problems and shown to provide global
convergence under the same set of assumptions. A closer inspection
indicates, however, that the mean-type rule requires slightly weaker
conditions than the max-type rule. This is also reflected by the 
two recent papers \cite{KanzowMehlitz2022,DeMarchi2023}. While
\cite{KanzowMehlitz2022} considers a nonmonotone proximal gradient
method with the max-type rule, the subsequent work \cite{DeMarchi2023}
investigates a nonmonotone proximal gradient method with the mean-type
rule. Both papers show that each accumulation point is a suitable
stationary point of the objective function $ \psi $, but 
\cite{KanzowMehlitz2022} requires a stronger (uniform continuity)
assumption, whereas \cite{DeMarchi2023} shows that the convergence 
theory for the monotone proximal gradient method can be adapted to the 
mean-type nonmonotone version without any extra condition.

The aim of this paper is to present a convergence and a rate-of-convergence 
result for the mean-type nonmonotone proximal gradient method under
the KL property, and without assuming a global Lipschitz assumption 
regarding $ f' $ or an a priori condition like the boundedness
of the iterates $ \xk $. To the best of our knowledge, this is the
first time that such results are shown without
any of these extra conditions. Our results are build on a combination of
ideas from the convergence theory in the monotone setting as in \cite{XiaoxiKanzowMehlitz2023},
recent contributions to mean-type nonmonotone proximal gradient methods in \cite{DeMarchi2023}
and a refinement of the (global) convergence theory in \cite{Qian-et-al-2024}
for our setting with weaker assumptions.
Unfortunately, the usual analysis
for monotone methods based on the KL property is already technical,
and the introduction of the nonmonotonicity complicates things even further.
Nevertheless, we feel it is worth going through this analysis taking
into account that nonmonotone methods often outperform their monotone
versions in practical applications, see \cite{DeMarchi2023,qian2022convergence} for numerical comparisons
in the context of proximal gradient methods. 
Taking this into account, we note
that the current paper is a purely theoretical work providing an
improved insight into the behaviour of a well-established method
for the solution of composite optimization problems.

The paper is organized in the following way. We first recall some
background material in Section~\ref{Sec:Background}. We then
state our nonmonotone proximal gradient method in 
Section~\ref{Sec:Algorithm} and recall some of its basic properties.
The corresponding convergence analysis is then presented in
Section~\ref{Sec:Convergence}. We close with some final remarks
in Section~\ref{Sec:Final}.

Notation: We write $ \langle x, y \rangle $ for the scalar product of
two elements $ x, y \in \X $, and $ \| x \| $ for the corresponding
norm of $ x \in \X $. The induced distance of a point $ x \in
\X $ to a nonempty set $ S \subseteq \X $ is denoted by
$ \dist(x, S) := \inf_{y \in  S} \| x - y \| $. The closed ball
around a given point $ x \in \X $ with radius $ r > 0 $ is
denoted by $ B_r (x) $.
Given a differentiable mapping $ f: \X \to \R $, we
write $ f'(x) $ for its derivative at $ x \in \mathbb{X} $. Finally,
$ \R $ denotes the set of real numbers, while $ \barR := (- \infty, 
+ \infty ] $ is the set of extended reals except that we exclude the
value $ - \infty $. Given an extended-valued function $ \theta: 
\X \to \barR $, we call $ \dom (\theta) := 
\{ x \in \X \mid \theta (x)  < \infty \} $
the \emph{domain} of $ \theta $. The function $ \theta $ is said to
be \emph{proper} if $ \dom (\theta) $ is nonempty. Throughout this
manuscript, we identify the dual space $ \X^* $ with $ \X $ itself.

\section{Background Material}\label{Sec:Background}

We first recall that a sequence $ \{ \xk \} \subseteq \X $
converges (locally) \emph{Q-linearly} to some limit 
$ x^* \in \X $ if there exists a constant $ c \in (0,1) $ 
such that 
\begin{equation*}
	\| \xkp - x^* \| \leq c \| \xk - x^* \|
\end{equation*}
holds for all $ k \in \N $ sufficiently large. Furthermore,
we say that $ \{ \xk \} $ converges \emph{R-linearly} to $ x^* $ if
\begin{equation*}
	\limsup_{k \to \infty} \| \xk - x^* \|^{1/k} < 1 
\end{equation*}
holds. Note that this property holds if there exist constants
$ \omega > 0 $ and $ \mu \in (0,1) $ such that $ \| \xk - x^* \| \leq 
\omega \mu^k $ for all $ k \in \N $ sufficiently large, i.e.,
if the sequence $ \| \xk - x^* \| $ is dominated by a Q-linearly
convergent null sequence.

We next recall some results from variational analysis and refer the
interested reader to the two monographs \cite{Mordukhovich2018,RockafellarWets2009}
for more details.

Given a proper, lower semicontinuous function $ \theta: \X \to
\barR $ and any $ x \in \dom (\theta) $, we call 
\begin{equation*}
	\hat\partial \theta (x) := \Big\{ v \in \X \, \Big|
	\liminf_{y \to x, y \neq x} \frac{\theta (y) - \theta (x) - 
	\langle v, y-x \rangle}{\| y- x \|} \geq 0 \Big\}
\end{equation*}
the \emph{regular or Fr\'echet subdifferential} of $ f $ at $ x $, 
whereas
\begin{equation*}
	\partial \theta (x) := \big\{ v \in \X \mid 
	\exists x^k,  v^k \in \X : \xk \to x, \theta (\xk) \to \theta (x),
	v^k \in \hat\partial \theta (\xk) \ \forall k \big\}
\end{equation*}
is called the \emph{limiting, Mordukhovich}, or 
\emph{basic subdifferential}
of $ f $ at $ x $. Recall that both sub\-differentials coincide with 
the convex subdifferential for convex functions $ \theta $, i.e., 
in this case, we have 
\begin{equation*}
	\hat\partial \theta (x) = 	\partial \theta (x) =
	\big\{ v \in \X \, \big| \, \theta (y) \geq \theta (x) + \langle 
	v, y-x \rangle \, \forall y \in \dom (\theta) \big\},
\end{equation*}
whereas for general nonconvex functions, only the inclusion
$ \hat\partial \theta (x) \subseteq \partial \theta (x) $ holds. 
These two subdifferentials have different properties, e.g., 
$ \hat\partial \theta (x) $ might be empty even for locally 
Lipschitz continuous functions, whereas $ \partial \theta (x) $ is 
known to be nonempty in this situation. Moreover, the limiting
subdifferential is \emph{robust} in the following sense: Given
a sequence $ \{ \xk \} $ converging to some limit $ x $ and a
corresponding sequence $ \{ v^k \} $ with $ v^k \in \partial \theta 
(\xk) $ for all $ k \in \N $ such that $ v^k \to v $ for some 
$ v \in \X $, then the inclusion $ v \in \partial \theta (x) $ holds.
This robustness property is highly important in the convergence theory
for many algorithms in the context of nonsmooth optimization, and simple
examples show that, in general, it is not shared by the Fr\'echet
subdifferential.

Given a lower semicontinuous function $ \theta : \X \to \barR $ and 
a local minimum $ x^* $ of $ \theta $, it follows immediately from 
the definition of the Fr\'echet subdifferential that
$ 0 \in \hat\partial \theta (x^*) $ holds. In particular, we therefore
get $ 0 \in \partial \theta (x^*) $. Any point $ x^* $ satisfying
this relation is called an \emph{M-stationary point} (M = Mordukhovich)
or simply a \emph{stationary point} of $ \theta $.

Now, coming back to our composite optimization problem where 
$ \psi = f + \phi $ is the sum of a continuously differentiable and
a lower semicontinuous function, the sum rule 
\begin{equation}\label{eq:sumrule}
	\partial \psi (x) = f'(x) + \partial \phi (x) \quad
	\forall x \in \dom (\phi)
\end{equation}
holds for the limiting subdifferential, cf.\ \cite{Mordukhovich2018}. In particular,
we therefore have a stationary point $ x^* $ of problem \eqref{Eq:ComOpt}
if and only if 
\begin{equation*}
	0 \in f'(x^*) + \partial \phi (x^*)
\end{equation*}
holds.

We finally introduce the Kurdyka–Łojasiewicz property which will play
a central role for our subsequent rate-of-convergence result. The 
following definition is a generalization of the classical 
one for nonsmooth functions, as introduced in \cite{AttouchBolteRedontSoubeyran2010,BolteDaniilidisLewis2007,BolteDaniilidisLewisShiota2007}.
Note that this KL property plays a central role in the local
convergence analysis of several algorithms for the solution of 
nonsmooth minimization problems, see \cite{attouch2009convergence,AttouchBolteSvaiter2013,BolteSabachTeboulle2014,BotCsetnek2016,BotCsetnekLaszlo2016,XiaoxiKanzowMehlitz2023,Ochs2018,OchsChenBroxPock2014} for a couple
of examples.

\begin{definition}\label{def:klproperty}
Let $g: \mathbb{X} \to \overline{\mathbb{R}}$ be lower semicontinuous. We say that $g$ satisfies the {\normalfont Kurdyka–Łojasiewicz (KL) property} at $x^* \in \{x \in \mathbb{X} \, | \, \partial g(x) \neq \emptyset\}$ if there exists a constant $\eta > 0$, a neighborhood $U \subset \mathbb{X}$ of $x^*$, and a continuous and concave function $\chi : [0, \eta] \to [0, \infty)$, called {\normalfont desingularization function}, which is continuously differentiable on $(0, \eta)$ and satisfies $\chi(0)=0$ and $\chi'(t) > 0$ for all $t \in(0, \eta)$, such that the so-called KL inequality 
\begin{equation*}
	\chi' \big( g(x) - g(x^*) \big) \dist \big( 0, \partial g(x)
	\big) \geq 1
\end{equation*}
holds for all $x \in U \cap \{x \in \mathbb{X} \, | 
\, g(x^*) < g(x) < g(x^*) + \eta\}$.
\end{definition}

\noindent
Note that there exist whole classes of functions where the KL property
is known to hold with the corresponding desingularization function
$ \chi (t) := c t^{\kappa} $ for some $ \kappa \in (0, 1] $ and a 
constant $ c > 0 $, where $ \kappa $ is called the \emph{KL exponent},
see \cite{BolteDaniilidisLewisShiota2007,Kurdyka1998}. This KL property looks somewhat artificial at a first
glance, but turns out to be a very useful and general tool for 
proving global convergence of the entire sequence as well as 
local rate-of-convergence results. Moreover, for $ g $ being a convex
function, the KL property is known to be equivalent to several
other concepts like a quadratic growth condition, a proximal error bound,
or the metric subregularity condition, see \cite{Ye_et_al_2021} 
for more details and corresponding references.

We finally restate a technical result from \cite[Lemma 1]{Aragon2018}
which will be used in order to simplify our final rate-of-convergence 
result.

\begin{lemma}\label{Lem:Aragon}
Let $ \{ s_k \} \subseteq [0, \infty )$ be any monotonically decreasing sequence satisfying $ s_k \to 0 $ and 
\begin{equation*}
	s_k^\alpha \leq \beta \big( s_k - s_{k+1} \big) \quad
	\text{for all $ k $ sufficiently large},
\end{equation*}
where $ \alpha, \beta > 0 $ are suitable constants. Then the following
statements hold:
\begin{itemize}
	\item[(a)] If $ \alpha \in (0,1] $, the sequence $ \{ s_k \} $
	converges linearly to zero with rate $ 1 - \frac{1}{\beta} $.
	\item[(b)] If $ \alpha > 1 $, there exists a constant $ \eta > 0 $
	such that 
	\begin{equation*}
		s_k \leq \eta k^{- \frac{1}{\alpha - 1}} \quad
		\text{for all $ k $ sufficiently large}.
	\end{equation*}
\end{itemize}
\end{lemma}

\section{Nonmonotone Proximal Gradient Method}\label{Sec:Algorithm}

This section gives a precise presentation of our nonmonotone 
proximal gradient method and provides some of its basic properties.
We first state the assumptions that we suppose to hold throughout
our theoretical investigation of the method.

\begin{assumption}\label{genas}
Assume:
\begin{itemize}
	\item[(a)] $\psi$ is bounded from below on $\dom (\phi )$,
	\item[(b)] $\phi$ is bounded from below by an affine function,
	\item[(c)] $ f' $ is locally Lipschitz continuous.
\end{itemize}
\end{assumption}

\noindent
Note that the first condition is very reasonable since otherwise
the given composite optimization problem \eqref{Eq:ComOpt} would be
unbounded from below. The second condition essentially guarantees 
that the proximal gradient subproblems 
\eqref{Eq:k-subproblem} have a solution (not necessarily
a unique one) since this implies that eventually the quadratic
term dominates the behaviour of the corresponding function. Finally,
we stress that the
local Lipschitz condition is equivalent to $ f' $ being globally
Lipschitz continuous on \emph{compact} sets, and that this local
Lipschitz property is a much weaker condition than 
the usual global Lipschitz assumption, e.g., the exponential
function, the natural logarithm, and all polynomials of degree higher
than two are locally Lipschitz, but not globally Lipschitz on their 
respective domains.

Before we present our algorithm, recall that the basic iteration 
of the proximal gradient method is given by \eqref{Eq:k-subproblem}
or, equivalently, by \eqref{Eq:ProxIteration}.
Observe that the next iterate $ \xkp $ depends on the choice of 
the stepsize parameter $ \gamma_k $, though, for simplicity, this is
not made explicit in our notation. A central observation is that,
for sufficiently small $ \gamma_k > 0 $, the corresponding solution
$ \xkp $ satisfies the property
\begin{equation}\label{Eq:monotone-criterion}
	\psi (\xkp) \leq \psi (\xk) - \omega\frac{1}{2 \gamma_k} 
	\| \xkp - \xk \|^2
\end{equation}
for any given parameter $ \omega \in (0,1) $, provided that $ \xk $
is not already an M-stationary point of $ \psi $, see, e.g., 
\cite{DeMarchi2023,KanzowMehlitz2022} 
for a formal proof. This implies that the sequence
$ \{ \psi (x^k) \} $ is monotonically decreasing. Now, suppose that
we have a reference value $ \Rvalue_k \geq \psi (\xk) $. Then, 
of course, any $ \xkp $ satisfying the monotone criterion 
\eqref{Eq:monotone-criterion} also satisfies the inequality
\begin{equation}\label{Eq:nonmonotone-criterion}
	\psi (\xkp) \leq \Rvalue_k - \omega \frac{1}{2 \gamma_k} 
	\| \xkp - \xk \|^2,
\end{equation}
but this condition might already be satisfied for a larger choice
of $ \gamma_k $, hence, resulting in a larger step, which is the
reason why nonmonotone methods often outperform their monotone
counterparts in practical applications. 

In order to obtain
suitable (global) convergence results, the reference value $ \Rvalue_k $
has to be chosen in a careful way. One popular choice is due to
Grippo et al.\ \cite{GrippoLamparielloLucidi1986}, 
where $ \Rvalue_k := \max \{ \psi (x^l) \mid l = k, k-1, \ldots, k- l_k \} $ for some given $ l_k \in \N $. We call this
strategy the \emph{max-rule} since $ \Rvalue_k $ is defined as the maximum
function value over the last few iterates, say, the last ten points.
In our \Cref{Alg:NonmonotoneProxPoint}, however, we use
the technique introduced by Zhang and Hager \cite{ZhangHager2004}, 
where $ \Rvalue_{k+1} $
is computed as a convex combination of the previous reference value
$ \Rvalue_k $ and the new function value $ \psi (\xkp) $. We 
therefore call this the \emph{mean-rule}. The details are given in Algorithm~\ref{Alg:NonmonotoneProxPoint}. Note that, to be closer
to the original version of this nonmonotonicity criterion, we replace
the constant $ \omega \in (0,1) $ used in  
\eqref{Eq:monotone-criterion} and \eqref{Eq:nonmonotone-criterion}
by a sequence of some values $ 1-\alpha_k \in (0,1) $.

\begin{algorithm}[Nonmonotone Proximal Gradient Method]\leavevmode
	\label{Alg:NonmonotoneProxPoint}
\begin{algorithmic}[1]
	\Require $x^0 \in \dom (\phi)$, $\varepsilon > 0$, $0<\gamma_\mathrm{min} \leq \gamma_\mathrm{max} < \infty$, $0 < \alpha_\mathrm{min} \leq \alpha_\mathrm{max} < 1$, $0 < \beta_{min} \leq \beta_\mathrm{max} < 1$, $\pmin \in (0, 1]$.
	\State set $\Rvalue_0 := \psi(x_0)$.
	\For{$k = 0, 1, 2, \dots$}
	\State choose $\gamma_k \in [\gamma_\mathrm{min}, \gamma_\mathrm{max}]$.
	\State compute $\xkp \in \textrm{Prox}_{\gamma_k \phi} (x^k - \gamma_k f'(x^k))$. \label{step:xkp}
	\If{$\| \frac{1}{\gamma_k}(\xkp - x^k) - f'(\xkp) + f' (x^k) \| \leq \varepsilon$}
	\State \Return $x^* := \xkp$
	\EndIf
	\State choose $\alpha_k \in [\alpha_\mathrm{min}, \alpha_\mathrm{max}]$ and $\beta_k \in [\beta_\mathrm{min}, \beta_\mathrm{max}]$
	\If{$\psi(x^{k+1}) > \Rvalue_k - \frac{1-\alpha_k}{2\gamma_k} \|\xkp - x^k\|^2$}
	\State set $\gamma_k := \beta_k \gamma_k$ and go back to step \ref{step:xkp}.
	\EndIf
	\State choose $p_{k+1} \in [\pmin, 1]$ and set $\Rvalue_{k+1} := (1-p_{k+1})\Rvalue_k + p_{k+1} \psi(\xkp)$.
	\EndFor
\end{algorithmic}
\end{algorithm}

\noindent
We note that our convergence theory implicitly assumes
that Algorithm~\ref{Alg:NonmonotoneProxPoint} generates an infinite
sequence. In particular, we assume that the practical termination 
criterion included into line 5 of Algorithm~\ref{Alg:NonmonotoneProxPoint}
never holds. This test can be interpreted as a measure for $ \xkp $
being close to an M-stationary point, see \cite{KanzowMehlitz2022} for 
further details. Note that this also implies that $ \xkp \neq \xk $
holds for all $ k $ since otherwise the current iterate $ \xk $ is
both an M-stationary point of $ \psi $ and also a point satisfying
the termination criterion from line 5 of Algorithm~\ref{Alg:NonmonotoneProxPoint}.

The following properties can be verified for Algorithm~\ref{Alg:NonmonotoneProxPoint}, see \cite{DeMarchi2023,KanzowMehlitz2022} for
the details.

\begin{lemma}\label{Lem:KnownProperties}
Let Assumption~\ref{genas} be satisfied.
Then the following statements hold for each sequence
$ \{ \xk \} $ generated by \Cref{Alg:NonmonotoneProxPoint}:
\begin{itemize}
	\item[(a)] $ \Rvalue_k \geq \psi (x^k) $ for all $ k \in \N $.
	\item[(b)] The sequence $ \{ \Rvalue_k \} $ is monotonically
	   decreasing.
	\item[(c)] The inner loop for the stepsize calculation $ \gamma_k $
	   is finite for each $ k $ (provided that $ \xk $ is not
	   already an M-stationary point).
	\item[(d)] $ \| \xkp - \xk \| \to 0 $ for $ k \to 
	   \infty $.
\end{itemize}
\end{lemma}

\noindent
The formal proof of this result can be found in \cite{DeMarchi2023,KanzowMehlitz2022}. We note, however,
that statements (a), (b), and (c)
are relatively simple to verify or standard observations. Part (d)
is then a direct consequence of these statements, in fact, from the
computation of $ \Rvalue_k $, the acceptance criterion for the
stepsize $ \gamma_k $, and the updates of the 
corresponding parameters, we obtain 
\begin{align}
	\Rvalue_{k+1} & = ( 1 - p_{k+1}) \Rvalue_k + p_{k+1} \psi (\xkp)
	\nonumber \\
	& \leq ( 1 - p_{k+1}) \Rvalue_k + p_{k+1} \Big( \Rvalue_k -
	\frac{1 - \alpha_k}{2 \gamma_k} \| \xkp - \xk \|^2 \Big)\label{Eq:Rupdate} \\
	& \leq \Rvalue_k - \pmin \frac{1 - \alpha_{\max}}{2 \gamma_{\max}} \| \xkp - \xk \|^2. \nonumber
\end{align}
Rearranging these terms yields
\begin{equation}\label{Eq:xdiff-conv}
	\Rvalue_{k+1} - \Rvalue_k \leq - \pmin \frac{1 - \alpha_{\max}}{2 \gamma_{\max}}
	\| \xkp - \xk \|^2 \leq 0
\end{equation}
for all $ k \in \N $. Now, since $ \psi $ is bounded from below
by \Cref{genas}, we obtain from \Cref{Lem:KnownProperties} (a)
that the sequence $ \{ \Rvalue_k \} $ is also bounded from below.
In view of \Cref{Lem:KnownProperties} (b), it follows that this
sequence is convergent. Consequently, the left-hand side from
\eqref{Eq:xdiff-conv} converges to zero. Hence, statement (d)
of \Cref{Lem:KnownProperties} is a consequence of
\eqref{Eq:xdiff-conv} and the sandwich theorem.

We next summarize the main global convergence properties of 
\Cref{Alg:NonmonotoneProxPoint}, the corresponding proofs can be found in
\cite{DeMarchi2023}.

\begin{theorem}\label{Thm:GlobConv}
Let Assumption~\ref{genas} be satisfied and $ x^* $ be an accumulation
point of a sequence $ \{ \xk \} $
generated by Algorithm~\ref{Alg:NonmonotoneProxPoint}.
Then the following statements hold:
\begin{itemize}
	\item[(a)] $ x^* $ is an M-stationary point of $ \psi $.
	\item[(b)] The sequence $ \{ \psi (\xk) \} $ converges to 
	   $ \psi (x^*) $.
	\item[(c)] The sequence $ \{ \Rvalue_k \} $ converges monotonically
	   to $ \psi (x^*) $.
\end{itemize}
\end{theorem}

\noindent 
Note that the central statement of \Cref{Thm:GlobConv} is 
assertion (a), the 
corresponding technical proof in \cite{DeMarchi2023} follows the 
ideas from \cite{KanzowMehlitz2022}. The other two statements are
easier to verify. In fact, statement (c) is mainly a consequence
of the observation from \Cref{Lem:KnownProperties} (b) that the
sequence $ \{ \Rvalue_k \} $ is monotonically decreasing, and
the (usually nonmonotone) convergence of the sequence $ \{ \psi (\xk) \} $
can then be derived from this observation together with the 
update of $ \Rvalue_{k+1} $ in \Cref{Alg:NonmonotoneProxPoint}.

\section{Convergence Theory}\label{Sec:Convergence}

Theorem~\ref{Thm:GlobConv} shows very satisfactory global
(subsequential) convergence
properties under fairly mild assumptions regarding the two functions
$ f $ and $ \phi $. The aim of this section (and of this paper)
is to show that, given some accumulation point of a sequence
generated by \Cref{Alg:NonmonotoneProxPoint} such that the 
KL property holds at this point, then the entire sequence converges
to this limit point and, in addition, has 
very favourable rate-of-convergence properties.
In other words, we obtain essentially the same convergence properties for 
our nonmonotone proximal gradient method as those known for
its monotone counterpart from \cite{XiaoxiKanzowMehlitz2023}.
Once again, we stress that this result holds without
any convexity of $ f $ or $ \phi $, without a global Lipschitz 
assumption regarding $ f' $, without any explicit knowledge of
a local Lipschitz constant, without an a priori assumption that 
the sequence $ \{ \xk \} $ remains bounded, and with $ \phi $ being merely
lower semicontinuous.

The corresponding convergence theory requires some technical results
which are inspired by the corresponding ones in the papers
\cite{DeMarchi2023,XiaoxiKanzowMehlitz2023,Qian-et-al-2024} and modified in 
a suitable way to deal with the above general setting.

To this end, we begin with the following result which is the
nonmonotone counterpart of \cite[Lemma 4.1]{XiaoxiKanzowMehlitz2023}, see
also \cite[Corollary 4.5]{DeMarchi2023}.

\begin{lemma}\label{lemma:bs}
Let \Cref{genas} hold, $\{x^k\}$ be any sequence generated
by Algorithm~\ref{Alg:NonmonotoneProxPoint}, and let $x^*$ be an
accumulation point of $\{x^k\}$. Then, for each $\rho > 0$, there exists a constant $\underline{\gamma}_\rho > 0 $ such that $\gamma_k \geq \underline{\gamma}_\rho$ holds for all $k \in \mathbb{N}$
satisfying $x^k \in B_\rho(x^*)$.
\end{lemma}

\begin{proof}
Recall that the stepsizes $\{\gamma_k\}$ are well-defined in view 
of \Cref{Lem:KnownProperties} (c). Let $\rho > 0$ be fixed. Since 
$f'$ is locally Lipschitz continuous, it is globally Lipschitz 
continuous on $B_{2 \rho}(x^*)$ with Lipschitz constant denoted by $L_{2 \rho}$. Since $x^*$ is an accumulation point, there are infinitely many iterates belonging to $B_\rho(x^*)$. 

Assume, by contradiction, that there exists a subsequence
$ \{ \xk \}_K $ with $ \xk \in B_{\rho}(x^*) $ for all $ k \in K $
and such that $ \{ \gamma_k\}_K$ is not bounded away from $0$. 
Without loss of generality, we may assume $\gamma_k \to_K 0$ and 
$x^k \to_K \bar x$ for some $\bar x \in B_{\rho}(x^*)$, and that the acceptance criterion for the
computation of the stepsize $ \gamma_k $ is violated in the previous 
iteration of the inner loop. For the corresponding trial stepsize $\hat \gamma_k := \gamma_k / \beta_k$, we then have $\gamma_k / \beta_{\max} \leq \hat \gamma_k \leq \gamma_k / \beta_{\min}$. This shows that we also
have $\hat \gamma_k \to_K 0$. 

Then the corresponding trial vector $\xhkp$, i.e., the solution of 
the subproblem
\begin{equation}\label{Eq:trial-subproblem}
	\min_x \ f(\xk) + \big\langle f'(\xk), x - \xk 
	\big\rangle + \frac{1}{2 \hat \gamma_k}
	\| x - \xk \|^2 + \phi (x), \quad x \in \X,
\end{equation}
does not satisfy the stepsize condition with associated parameter $\hat \alpha_k \in [\alpha_{\min}, \alpha_{\max}]$, i.e.,
\begin{equation}\label{eq:stepsizecond}
	\psi(\xhkp) > \Rvalue_k- \frac{1-\hat \alpha_k}{2 \hat \gamma_k} \|\xhkp - x^k\|^2 \geq \Rvalue_k - \frac{1- \alpha_{\min}}{2 \hat \gamma_k} \|\xhkp - x^k\|^2.
\end{equation}
Note that \eqref{eq:stepsizecond} in combination with 
$ \Rvalue_k \geq \psi (x^k) $, cf.\ \Cref{Lem:KnownProperties} (a),
implies, in particular, that we have $ \xk \neq \xhkp $. 
Moreover, since $\xhkp$ is a solution of \eqref{Eq:trial-subproblem}, 
we have
\begin{equation}\label{eq:subprob}
	\big\langle f'(x^k), \xhkp - x^k \big\rangle + \frac{1}{2 \hat \gamma_k} \|\hat x^{k+1} - x^{k}\|^2 + \phi(\xhkp) \leq \phi(x^{k}).
\end{equation}
Using the Cauchy-Schwarz inequality and the fact that
$\psi(\xk) \leq \Rvalue_k \leq \Rvalue_0 $ for all $k$ by
\Cref{Lem:KnownProperties} (a), (b), we obtain
\begin{align*}
	\frac{1}{2 \hat \gamma_k} \| \xhkp - \xk \|^2 &\leq \| f'(\xk)\| \|\xhkp - \xk\| + \phi(\xk) - \phi(\xhkp)\\
	&= \| f'(\xk)\| \|\xhkp - \xk\| + \psi(\xk) - f(\xk) - \phi(\xhkp)\\
	&\leq \| f'(\xk)\| \|\xhkp - \xk\| + \Rvalue_0 
	- f(\xk) - \phi(\xhkp).
\end{align*}
Using the continuous differentiability of $f$ and the boundedness 
condition from \Cref{genas} for $\phi$, we claim that the above 
inequality implies $\xhkp - \xk \to_K 0$: Assume, by contradiction, 
that $\{\| \xhkp - \xk \|\}_{k \in K}$ would be unbounded, then the left-hand side would grow more rapidly than the right-hand side. If $\{\| \xhkp - \xk \|\}_{k \in K}$ remains bounded but (at least on a subsequence) staying away from zero, then the right-hand side is bounded but the left-hand side is unbounded as we have $\hat \gamma_k \to_K 0$. Hence,
we have $\xhkp - \xk \to_K 0$.

Taking this into account and using the fact that $\bar x \in B_\rho(x^*)$,
it follows that
\begin{equation}\label{Eq:Ball-inclusions}
	\xk \in B_{\rho}(x^*) \subset B_{2 \rho}(x^*) \quad \text{and} \quad 
	\xhkp \in B_{2 \rho}(x^*) \quad \text{for sufficiently large }
	k \in K.
\end{equation}
We will exploit this observation later.

Using the differential mean-value theorem, there exists a point 
$\xi^k$ on the line segment between $\xk$ and $\xhkp$ such that
\begin{align*}
	\psi(\xhkp) - \psi(\xk) &= f(\xhkp) + \phi(\xhkp) - f(\xk) - \phi(\xk)\\
	& = \big\langle f'(\xi^k), \xhkp - \xk \big\rangle + \phi(\xhkp) - \phi(\xk).
\end{align*}
Substituting $\phi(\xhkp) - \phi(\xk)$ into \eqref{eq:subprob} yields
\begin{equation*}
   \big\langle f'(\xk) - f'(\xi^k), \xhkp - \xk \big\rangle + \frac{1}{2 \hat \gamma_k} \| \xhkp - \xk \|^2 + \psi(\xhkp) - \psi(\xk) \leq 0.
\end{equation*}
Now, we obtain
\begin{align*}
   \frac{1}{2 \hat \gamma_k} \| \xhkp - \xk \|^2 &
   \leq - \big\langle f'(\xk) - f'(\xi^k), \xhkp - \xk \big\rangle + \psi(\xk) - \psi(\xhkp)\\
	& \leq - \big\langle f'(\xk) - f'(\xi^k), \xhkp - \xk \big\rangle + \psi(\xk) - \Rvalue_{k} + \frac{1-\alpha_\text{min}}{2 \hat \gamma_k} \| \xhkp - \xk \|^2\\
	& \leq \| f'(\xk) - f'(\xi^k)\| \| \xhkp - \xk \| + \frac{1-\alpha_\text{min}}{2 \hat \gamma_k} \| \xhkp - \xk \|^2,
\end{align*}
where the second inequality is due to \eqref{eq:stepsizecond} and the 
final estimate follows from the Cauchy-Schwarz inequality together with  $\Rvalue_{k} \geq \psi(\xk)$, cf.\ \Cref{Lem:KnownProperties} (a). As $\xhkp \neq \xk$ in view of our previous discussion, the resulting
expression can be simplified to
\begin{equation*}
   \frac{\alpha_\text{min}}{2 \hat \gamma_k} \| \xhkp - \xk\| \leq \| f'(\xk) - f'(\xi^k) \|.
\end{equation*}
Since $\xi^k$ is on the line connecting $\xk$ and $\xhkp$, it follows 
from \eqref{Eq:Ball-inclusions} that also $\xi^k \in B_{2 \rho}(x^*)$ 
holds for all $k \in K$ sufficiently large. Thus, using the Lipschitz continuity of $ f' $ on $B_{2 \rho}(x^*)$, we conclude
\begin{equation*}
   \frac{\alpha_\text{min}}{2 \hat \gamma_k} \| \xhkp - \xk\| \leq L_{2\rho} \| \xk - \xi^k \| \leq L_{2 \rho} \|\xk - \xhkp\|.
\end{equation*}
As $\xhkp \neq \xk$, we get $ \hat \gamma_k \geq \frac{\alpha_\text{min}}{2 L_{2\rho}} $ for all $ k \in K $ sufficiently large.
This, in turn, implies that the corresponding subsequence $\{\gamma_k\}_K$
is also bounded from below by a positive constant, but this 
contradicts our assumption. Altogether, this completes the proof. 
\end{proof}

\noindent
For the remaining part, assume that our objective function $\psi$ satisfies the KL property at a given accumulation point $x^*$. Let $\eta > 0$ be the corresponding constant and $\chi$ the associated desingularization function from \Cref{def:klproperty}. Furthermore, we denote by $\{x^k\}_{k \in K}$ a subsequence converging to $x^*$. 

The subsequent theory requires some further constants and indices which
will be introduced here and which will be used throughout the remaining
part of this section. To this end, we first note that, in view
of \Cref{Lem:KnownProperties} (d), there exists an index $\hat k
\in \N $ such that 
\begin{equation}\label{eq:supnorm}
   \sup_{k \geq \hat k} \|x^{k+1} - x^k\| \leq \eta.
\end{equation}
Define 
\begin{equation*}
   \rho := \eta + \frac{1}{2} \quad \text{and} \quad
   C_\rho := B_\rho(x^*) \cap \big\{ x \in \X \, \big| \, 
   \psi(x) \leq \Rvalue_0 \big\}.
\end{equation*}   
Let $L_\rho$ be a global Lipschitz constant of $f'$ on $C_\rho$. By \Cref{lemma:bs}, there exists a constant $\underline{\gamma}_\rho > 0$ 
such that 
\begin{equation*}
   \gamma_k \geq \underline{\gamma}_\rho \quad \text{for all } k
   \text{ with } \xk \in C_\rho.
\end{equation*}
Following mostly \cite{Qian-et-al-2024}, we also introduce
the following notation: 
\begin{itemize}
	\item $m := \min \big\{ l \in \N \, \big| \, (1-\sqrt{1-p_{\min}}) \sqrt{l} \geq (1+\sqrt{1-p_{\min}}) \big\} $,
	\item $l(k) := k + m - 1$, 
	\item $\Xi_{k-1} := \sqrt{\Rvalue_{k-1} - \Rvalue_{k}}$ for $k \in \N $
	   and
	\item $\Delta_{i, j} := \chi \big(\Rvalue_i- \psi(x^*) \big) - 
	   \chi \big( \Rvalue_j - \psi(x^*) \big)$.
\end{itemize}
Note that the index $ m $ is obviously uniquely defined since the 
left-hand side of the inequality eventually becomes larger than
the constant on the right-hand side. Furthermore, note that
the difference $ l(k)- k = m-1 $ is a constant number for all $ k \in \N $,
in particular, this difference does not increase to infinity for
$ k \to \infty $. This simple observation plays some role in the 
subsequent convergence analysis since it guarantees that certain
sums are always taken over a finite (fixed) number of terms only. Moreover,
we note that $\Xi_{k-1} \geq 0$ holds for all $ k \in \N $ by
the monotonicity property of the sequence $ \{ \Rvalue_k \} $ from
\Cref{Lem:KnownProperties} (b). Finally, we also have
$\Delta_{i, j} \geq 0$ for all $j \geq i$ by monotonicity of $\chi$
in combination with  \Cref{Lem:KnownProperties} (b) once again.

We further introduce the two index sets
\begin{equation*}
	K_1 := \big\{ k \in \N \, \big| \, \psi(x^k) \leq \Rvalue_{k+m}
	\big\},
\end{equation*}
and
\begin{equation*}
	K_2 := \big\{k \in \N \, \big| \, \psi(x^k) > \Rvalue_{k+m} \big\}
\end{equation*}
depending on the previously introduced number $ m $.
To simplify the notation, we define the constant 
\begin{equation*}
	a:= \frac{1-\alpha_{\max}}{2 \gamma_{\max}} > 0.
\end{equation*}
Using the update rule for $\Rvalue_k$ together with the acceptance
criterion for the step size $ \gamma_k $, we obtain
\begin{equation*}
   \Rvalue_k \leq
	\Rvalue_{k-1} - \frac{1-\alpha_{k-1}}{2 \gamma_{k-1}} p_k \|\xk - \xkm\|^2 \leq \Rvalue_{k-1} - a \pmin \|\xk - \xkm\|^2,
\end{equation*}
cf.\ \eqref{Eq:Rupdate}.
Therefore,
\begin{equation}\label{eq:normxi}
   \sqrt{a \pmin} \|\xk -\xkm\| \leq \Xi_{k-1}.
\end{equation}
The following results and proofs are motivated by the corresponding
analysis in \cite{Qian-et-al-2024}. However, to avoid the a priori 
assumption that the sequence generated by \Cref{Alg:NonmonotoneProxPoint}
is bounded, we need to modify the arguments to some extend, using
ideas from \cite{XiaoxiKanzowMehlitz2023,KanzowMehlitz2022}.

\begin{lemma}\label{Lem:alpha}
Define the constant $ \hat{c} := \frac{\sqrt{\pmin}}{2\sqrt{a}} \Big(\frac{1}{\gammag} + L_\rho\Big) $. Then there
exists a sufficiently large index $k_0 \in K$ such that
\begin{equation}
	\alpha := \|x^{k_0-1} - x^*\| + \frac{4}{\sqrt{a p_{\min}}} \sum_{j = k_0}^{l(k_0)} \Xi_{j-1} + \frac{2 \hat c}{\sqrt{a p_{\min}}} \sum_{j = k_0}^{l(k_0)} \chi \big( \Rvalue_j - \psi(x^*) \big)
\end{equation}
satisfies $\alpha < \frac{1}{2}$ and $B_\alpha(x^*) \subset U$, where $U$ is the neighborhood of $x^*$ from the KL property in \Cref{def:klproperty}.
\end{lemma}

\begin{proof}
First recall that $ l(k) - k = m - 1 $ is a fixed constant, hence,
the number of terms within each of the summations is fixed, 
independent of $ k $. Therefore, it is easy to see that each term
on the right-hand side can be made arbitrarily small: For the first term, this follows from  $x^k \to_K x^*$ together with the fact
that $ \| \xk - \xkm \| \to 0 $ in view of \Cref{Lem:KnownProperties} (d). For the second term, recall that $\Rvalue_k \to \psi(x^*)$ monotonically by \Cref{Thm:GlobConv} (c). For the final term, note again that $\Rvalue_k \to \psi(x^*)$ monotonically and that, by definition, the desingularization function $\chi$ is continuous at the origin.
\end{proof}

\noindent
Note that, in principle, \Cref{Lem:alpha} holds for an arbitrary
constant $ \hat{c} > 0 $. However, the index $ k_0 $ depends on
this constant $ \hat{c} $. Since the previous result will later
be applied to the particular choice of $ \hat{c} $ from 
\Cref{Lem:alpha} with the corresponding index $ k_0 $, the previous
result is formulated for this particular value of $ \hat{c} $.

The following result is the nonmonotone counterpart of \cite[Lemma 4.4]{XiaoxiKanzowMehlitz2023}. Regarding its assumption, we 
recall from \Cref{Thm:GlobConv} that the sequence 
$ \{ \Rvalue_k \} $ converges monotonically to the function value
$ \psi (x^*) $ at the accumulation point $ x^* $, hence, we eventually
have $ \Rvalue_k < \psi (x^*) + \eta $, where $ \eta > 0 $
denotes the constant from \Cref{def:klproperty}.

\begin{lemma}\label{distbound}
Under the conditions specified above, we have 
\begin{equation}
	\dist \big( 0, \partial \psi(\xkp) \big) 
	\leq \Big(\frac{1}{\gammag} + L_\rho \Big) \|\xkp - \xk\|
\end{equation}
for all sufficiently large $k \in \N $ such that $\xk \in B_\alpha(x^*)$ holds.
\end{lemma}

\begin{proof}
For all $k$, since $\xkp$ solves the proximal gradient subproblem
\eqref{Eq:k-subproblem}, we obtain
\begin{equation}
   0 \in f'(\xk) + \frac{1}{\gamma_{k}} (\xkp - \xk) + \partial \phi(\xkp)
\end{equation}
from the M-stationary condition together with the sum
rule \eqref{eq:sumrule} for the Mordukhovich subdifferential. This implies
\begin{equation}\label{eq:mstatimp}
   \frac{1}{\gamma_{k}} ( \xk - \xkp ) + f'(\xkp) - f'(\xk) \in
   f'(\xkp) + \partial \phi(\xkp) = \partial \psi(\xkp)
\end{equation}
using once again the sum rule from \eqref{eq:sumrule}. Take $k \in \mathbb{N}$ sufficiently large such that $\xk \in B_\alpha(x^*)$ and $k \geq \hat k$, where $\hat k$ is the index from \eqref{eq:supnorm}
and $ \alpha $ is the constant defined in \Cref{Lem:alpha}.
Since $\alpha \leq \rho$, we have $\gamma_k \geq \gammag$ by
\Cref{lemma:bs}. Now, the estimate
\begin{equation*}
   \| \xkp - x^* \| \leq \| \xkp - \xk \| + \| \xk - x^* \| \leq \eta + \alpha \leq \rho
\end{equation*}
shows that $\xk, \xkp \in C_\rho$. Therefore, we get
\begin{equation*}
   \| f'(\xkp) - f'(\xk) \| \leq L_\rho \| \xkp - \xk \|.
\end{equation*}
Using the bound on $\gamma_k$ and \eqref{eq:mstatimp} gives
\begin{align*}
   \dist \big( 0, \partial \psi(\xkp) \big) & \leq \Big\| \frac{1}{\gamma_k} (\xk - \xkp) + f'(\xkp) - f'(\xk) 
   \Big\| \\
	& \leq \frac{1}{\gamma_k} \|\xk - \xkp \| + L_\rho \| \xkp - \xk \| \\
	& \leq \Big(\frac{1}{\gammag} + L_\rho\Big) \| \xkp - \xk \|
\end{align*}
for all $k$ with $k \geq \hat k$ and $\xk \in B_\alpha(x^*)$.
\end{proof}

\noindent
We now present our final technical result, motivated by \cite{Qian-et-al-2024}.

\begin{lemma}\label{lemma:lemmaxi}
For all sufficiently large $k \in \mathbb{N}$ with $\Rvalue_k < \psi(x^*) + \eta$ and $\xkm \in B_\alpha(x^*)$, the following inequality holds:
\begin{equation}\label{eq:lemmaxi}
   \frac{1-\sqrt{1-p_{\min}}}{\sqrt{m}} \sum_{i = k}^{l(k)} \Xi_i \leq \left(1/2 + \sqrt{1-p_{\min}}\right) \Xi_{k-1} + \hat c \Delta_{k, k+m},
\end{equation}
where $\hat c $ denotes the constant from \Cref{Lem:alpha}.
\end{lemma}

\begin{proof}
First note that $x \mapsto \sqrt{x}$ is a concave function, thus the application of Jensen's inequality yields
\begin{equation}\label{eq:inequality}
   \frac{1-\sqrt{1-p_{\min}}}{\sqrt{m}} \sum_{i = k}^{l(k)} \Xi_i \leq \big( 1-\sqrt{1-\pmin} \big) \sqrt{\Rvalue_k - \Rvalue_{k+m}}.
\end{equation}
We now distinguish two cases.\smallskip

\noindent
\textbf{Case 1:} $k \in K_1$. We then have $ \psi (\xk) \leq 
\Rvalue_{k+m} $, which implies
\begin{align*}
   \Rvalue_k - \Rvalue_{k+m} &= (1-p_k) \Rvalue_{k-1} + p_k \psi(\xk) - \Rvalue_{k+m}\\
   & \leq (1-p_k) \Rvalue_{k-1} + p_k \Rvalue_{k+m} - \Rvalue_{k+m}\\
   & = (1-p_k) (\Rvalue_{k-1} - \Rvalue_{k+m})\\
	&\leq (1-\pmin) (\Rvalue_{k-1} - \Rvalue_{k+m})
	\qquad (\text{recall that } \Rvalue_{k-1} \geq \Rvalue_{k+m})\\
	&= (1-\pmin) (\Rvalue_{k-1} - \Rvalue_k + \Rvalue_k - \Rvalue_{k+m}).
\end{align*}
Using $\sqrt{x+y} \leq \sqrt{x} + \sqrt{y}$ for all $x, y \in \mathbb{R}_{\geq 0}$, we obtain
\begin{equation*}
   \big( 1- \sqrt{1- \pmin} \big) \sqrt{\Rvalue_k - \Rvalue_{k+m}} \leq \sqrt{1-\pmin} \Xi_{k-1}.
\end{equation*}
The statement therefore follows from \eqref{eq:inequality}. \smallskip

\noindent
\textbf{Case 2:} $k \in K_2$. We then have $\psi(x^*) \leq 
\Rvalue_{k+m} < \psi(\xk) \leq \Rvalue_{k} < \psi(x^*) + \eta$
by assumption. Using the KL property of $\chi$, we get
\begin{equation*}
   \chi' \big( \psi(\xk) - \psi(x^*) \big) 
   \dist \big( 0, \partial \psi(\xk) \big) \geq 1.
\end{equation*}
As $\xkm$ was assumed to be in $B_\alpha(x^*)$, by application of \Cref{distbound}, we have
\begin{equation}\label{eq:chi-star}
   \chi' \big( \psi(\xk) - \psi(x^*) \big) \geq \frac{1}{\Big(\frac{1}{\gammag} + L_\rho\Big) \|\xk - \xkm\|}
\end{equation}
(recall that $ \xk \neq x^{k-1} $ since \Cref{Alg:NonmonotoneProxPoint}
is assumed to generate an infinite sequence).
Using the properties of $\chi$, we now obtain
\begin{align*}
   \Delta_{k, k+m} &= \chi \big(\Rvalue_k- \psi(x^*) \big) -
    \chi \big( \Rvalue_{k+m} - \psi(x^*) \big) \\
	&\geq \chi \big( \psi(x^k)- \psi(x^*) \big) - 
	\chi \big( \Rvalue_{k+m} - \psi(x^*) \big)\\
	&\geq \chi' \big( \psi(\xk) - \psi(x^*) \big) 
	\big( \psi(\xk) - \Rvalue_{k+m} \big)\\
	&\geq \frac{\psi(\xk) - \Rvalue_{k+m}}{\Big(\frac{1}{\gammag} + L_\rho\Big) \|\xk - \xkm\|},
\end{align*}
where the first inequality results from the monotonicity of 
$ \chi $, the next one exploits the concavity of $ \chi $, and 
the final estimate exploits \eqref{eq:chi-star} together with the
fact that $ \psi (\xk) - \Rvalue_{k+m} > 0 $ in the case under
consideration. Thus, with \eqref{eq:normxi}, we get
\begin{equation*}
   \psi(\xk) - \Rvalue_{k+m} \leq \frac{2\hat c}{\pmin} \Xi_{k-1} \Delta_{k, k+m}
\end{equation*}
from the definition of $ \hat{c} $.
Similar to the first case, our aim is to bound 
the difference $\Rvalue_k - \Rvalue_{k+m}$. Using the fact that $\psi(\xk) \leq \Rvalue_{k-1}$ by the acceptance criterion for our 
stepsize computation as well as $p_k \geq \pmin$, we have
\begin{equation*}
	p_k \psi (\xk) + ( 1- p_k)\Rvalue_{k-1} \leq 
    p_{\min} \psi (\xk) + ( 1 - p_{\min}) \Rvalue_{k-1} .
\end{equation*}
Together with the definition of $ \Rvalue_k := (1-p_k)\Rvalue_{k-1} + p_k \psi(\xk) $, this yields
\begin{align*}
   \Rvalue_k - \Rvalue_{k+m} &= p_k \psi(\xk) + (1-p_k) \Rvalue_{k-1} - \Rvalue_{k+m} \\
	& \leq \pmin \psi(\xk) + (1-\pmin) \Rvalue´_{k-1} - \pmin \Rvalue_{k+m} - (1-\pmin) \Rvalue_{k+m}\\
	& = \pmin \big( \psi(\xk) - \Rvalue_{k+m} \big) + (1-\pmin) (\Rvalue_{k-1} - \Rvalue_{k+m})\\
	& \leq 2\hat c \Xi_{k-1} \Delta_{k, k+m}  + (1-\pmin) (\Rvalue_{k-1} - \Rvalue_{k+m})\\
	& = 2\hat c \Xi_{k-1} \Delta_{k, k+m}  + (1-\pmin) (\Rvalue_{k-1} - \Rvalue_k + \Rvalue_k - \Rvalue_{k+m}).
\end{align*}
Taking square roots on both sides and using
$\sqrt{x+y} \leq \sqrt{x} + \sqrt{y}$ for all $ x, y \in \mathbb{R}_{\geq 0}$, we obtain
\begin{equation*}
	\sqrt{\Rvalue_k - \Rvalue_{k+mn}}
	\leq \sqrt{2 \hat{c} \Xi_{k-1} \Delta_{k,k+m}} +
	\sqrt{1 - \pmin} \big( \sqrt{\Rvalue_{k-1} - \Rvalue_k} 
	+ \sqrt{\Rvalue_k - \Rvalue_{k+m}} \big),
\end{equation*}
so we have
\begin{equation*}
   \big( 1 - \sqrt{1 - \pmin} \big) \sqrt{\Rvalue_k - \Rvalue_{k+m}}
   \leq \sqrt{2 \hat{c} \Xi_{k-1} \Delta_{k,k+m}} +
   \sqrt{1 - \pmin} \Xi_{k-1}.
\end{equation*}
Exploiting the inequality $2\sqrt{xy} \leq x+y$ for all $x, y \in \mathbb{R}_{\geq 0}$, this yields
\begin{equation*}
   \big( 1-\sqrt{1-\pmin} \big) 
   \sqrt{\Rvalue_k - \Rvalue_{k+m}} \leq 
   \Big(\frac{1}{2} + \sqrt{1-\pmin} \Big) \Xi_{k-1} + \hat c \Delta_{k, k+m}.
\end{equation*}
In view of \eqref{eq:inequality}, this completes the proof.
\end{proof}

\noindent
The following result shows global convergence of the entire sequence $\{ \xk \}$ generated by the nonmonotone proximal gradient method to one of its accumulation points $x^*$, given that $\psi$ satisfies the KL property at this point. The proof follows the technique of the global convergence result in \cite{Qian-et-al-2024}. However, by using our previous results, we neither assume the a priori boundedness of the iterates $\{\xk\}$ 
nor do we require $f$ to satisfy a global Lipschitz condition.

\begin{theorem}\label{thm:convergence}
Let \Cref{genas} hold, let $\{\xk\}_K$ be a subsequence converging to some limit point $x^*$, and suppose that the KL property for $\psi$ holds at $x^*$. Then the entire sequence $\{\xk\}$ converges to $x^*$.
\end{theorem}

\begin{proof}
Let $k_0$ be the index from the definition of $\alpha$,
cf.\ \Cref{Lem:alpha}.
Without loss of generality, we may assume $k_0 \geq \hat k$, where $\hat k$ is the index from \eqref{eq:supnorm} and that $\Rvalue_{k_0} < \psi(x^*) + \eta$. We now claim that the following statements hold:
\begin{enumerate}[label=(\alph*)]
	\item for all $k \geq k_0-1$: $\xk \in B_\alpha(x^*)$, and
	\item for all $k \geq l(k_0)$:
	\begin{equation}\label{eq:statementb}
	\big( 1-\sqrt{1-\pmin} \big) 
	\sqrt{m} \sum_{j = l(k_0)}^k \Xi_j \leq \Big( \frac{1}{2} + \sqrt{1-\pmin} \Big) \sum_{j = k_0}^k \Xi_{j-1} + \hat c \sum_{j = k_0}^{l(k_0)} \chi \big( \Rvalue_j - \psi(x^*) \big),
	\end{equation}
\end{enumerate}
where $ \hat{c} $ denotes the constant from \Cref{Lem:alpha}.
We verify these two statements jointly by induction over $k$. For all $k \in \{k_0-1, \ldots, l(k_0)\}$, we obtain from \eqref{eq:normxi} 
and the definition of the constant $ \alpha $ in \Cref{Lem:alpha} that
\begin{align*}
   \| x^{k} - x^*\| & \leq \| x^{k_0-1} - x^* \| + 
   \sum_{j = k_0}^{k} \| x^j - x^{j-1} \| \\
   & \leq \| x^{k_0-1} - x^* \| + \sum_{j = k_0}^{l(k_0)} \| x^j - x^{j-1} \| \\
	& \leq \| x^{k_0-1} - x^* \| + \frac{1}{\sqrt{a \pmin}} \sum_{j = k_0}^{l(k_0)} \Xi_{j-1} \leq \alpha,
\end{align*}
which shows the first statement for $k = k_0-1, \dots, l(k_0)$. Now, we can apply \Cref{lemma:lemmaxi} for indices $k = k_0, \dots, l(k_0)$ and obtain
\begin{align*}
   \big( 1-\sqrt{1-\pmin} \big)
   \sqrt{m} \Xi_{l(k_0)}& \leq \frac{1 - \sqrt{1- \pmin}}{\sqrt{m}} \sum_{j=k_0}^{l(k_0)} \sum_{i = j}^{l(j)} \Xi_i \\
	&\leq \left(1/2 + \sqrt{1-p_{\min}}\right) \sum_{j=k_0}^{l(k_0)} \Xi_{j-1} + \hat c \sum_{j = k_0}^{l(k_0)} \Delta_{j, j+m}\\
	&\leq \left(1/2 + \sqrt{1-p_{\min}}\right) \sum_{j=k_0}^{l(k_0)} \Xi_{j-1} + \hat c \sum_{j = k_0}^{l(k_0)} 
	\chi \big( \Rvalue_j - \psi(x^*) \big),
\end{align*}
where the first inequality follows from the fact that the term
$ \Xi_{l(k_0)} $ occurs $ m $ times within the double sum on the
right-hand side, whereas the other expressions $ \Xi_i $ are
nonnegative, and the second inequality is obtained by summing \eqref{eq:lemmaxi} in \Cref{lemma:lemmaxi} from $k_0$ to $l(k_0)$. 
In the final estimate, we simply omit some nonpositive terms.
This shows that the second statement holds for $k=l(k_0)$.

Suppose that the first statement holds for all $j$ from $k_0-1$ to some $k \geq l(k_0)$ and that the second statement is true for $k \geq l(k_0)$. We first show that the second statement for $k$ implies the first statement for $k+1$. Using \eqref{eq:statementb}, we get
\begin{align*}
   & \big( 1-\sqrt{1-\pmin} \big) 
   \sqrt{m}\sum_{j = l(k_0)}^k \Xi_j \\
   & \hspace*{8mm} \leq \Big( \frac{1}{2} + \sqrt{1-\pmin} \Big) \sum_{j = k_0}^k \Xi_{j-1} + \hat c \sum_{j = k_0}^{l(k_0)} \chi \big( \Rvalue_j - \psi(x^*) \big) \\
   & \hspace*{8mm} \leq \Big( \frac{1}{2} + \sqrt{1-\pmin} \Big) 
   \Big( \sum_{j = k_0}^{l(k_0)} \Xi_{j-1} + \sum_{j = l(k_0)+1}^{k} \Xi_{j-1}\Big) + \hat c \sum_{j = k_0}^{l(k_0)} \chi \big( \Rvalue_j - \psi(x^*) \big) \\
   & \hspace*{8mm} \leq \Big( \frac{1}{2} + \sqrt{1-\pmin} \Big) 
   \Big( \sum_{j = k_0}^{l(k_0)} \Xi_{j-1} + \sum_{j = l(k_0)}^{k} \Xi_{j}\Big) + \hat c \sum_{j = k_0}^{l(k_0)} \chi \big( \Rvalue_j - \psi(x^*) \big).
\end{align*}
This implies 
\begin{align*}
   & \left( \Big(1-\sqrt{1-\pmin} \Big)\sqrt{m} - \Big(\frac{1}{2} + \sqrt{1-\pmin} \Big) \right) \sum_{j = l(k_0)}^k \Xi_j \\
   & \hspace*{8mm} \leq \Big( \frac{1}{2} + \sqrt{1-\pmin} \Big) \sum_{j = k_0}^{l(k_0)} \Xi_{j-1} + \hat c \sum_{j = k_0}^{l(k_0)}
   \chi \big( \Rvalue_j - \psi(x^*) \big).
\end{align*}
Noting that, by definition of $m$, it holds that 
$ \big( 1-\sqrt{1-\pmin} \big) \sqrt{m} - 
\big( 1/2 + \sqrt{1- \pmin} \big) \geq 1/2$ and that we obviously
have $1/2 + \sqrt{1-\pmin} \leq 3/2$, we get
\begin{equation*}
   \sum_{j=l(k_0)}^k \Xi_j \leq 3 \sum_{j=k_0}^{l(k_0)} \Xi_{j-1} + 2 \hat c \sum_{j = k_0}^{l(k_0)} \chi \big( \Rvalue_j - \psi(x^*) \big).
\end{equation*}
This implies
\begin{equation*}
   \sum_{j=k_0-1}^k \Xi_j = \sum_{j=k_0}^{l(k_0)} \Xi_{j-1} + \sum_{j=l(k_0)}^k \Xi_j \leq 4 \sum_{j=k_0}^{l(k_0)} \Xi_{j-1} + 2 \hat c \sum_{j = k_0}^{l(k_0)}\chi \big( \Rvalue_j - \psi(x^*) \big),
\end{equation*}
hence, we obtain from \eqref{eq:normxi} that
\begin{equation}\label{eq:normxkp}
\begin{aligned}
   \|x^{k+1} - x^*\| & \leq \| x^{k_0-1} - x^* \| + \sum_{j = k_0-1}^k \| x^{j+1} - x^j\| \leq \| x^{k_0-1} - x^* \| + \frac{1}{\sqrt{a \pmin}} \sum_{j=k_0-1}^k \Xi_j\\
	&\leq \| x^{k_0-1} - x^*\| + \frac{4}{\sqrt{a\pmin}} \sum_{j=k_0}^{l(k_0)} \Xi_{j-1} + \frac{2 \hat c}{\sqrt{a \pmin}} \sum_{j = k_0}^{l(k_0)}\chi \big( \Rvalue_j - \psi(x^*) \big).
\end{aligned}
\end{equation}
The expression on the right-hand side is precisely the constant
$ \alpha $ from \Cref{Lem:alpha}. Thus, the first statement holds 
for $k+1$.

We next verify the second part for $k+1$. Since we know that $x^j \in B_\alpha(x^*)$ is true for all $j \in \{k_0-1, \dots, k+1\}$ by our induction hypothesis, we again apply \Cref{lemma:lemmaxi} and sum over \eqref{eq:lemmaxi}, now from $k_0$ to $k+1$. This yields
\begin{align*}
   \big( 1-\sqrt{1-\pmin} \big) \sqrt{m} 
   \sum_{j=l(k_0)}^{k+1} \Xi_j & 
   \leq \frac{1-\sqrt{1-\pmin}}{\sqrt{m}} \sum_{j=k_0}^{k+1} \sum_{i=j}^{l(j)} \Xi_i\\
	&\leq \big( 1/2 + \sqrt{1-\pmin} \big)
	\sum_{j=k_0}^{k+1} \Xi_{j-1} + \hat c \sum_{j =k_0}^{k+1} \Delta_{j, j+m}\\
	&\leq \big( 1/2 + \sqrt{1-\pmin} \big)
	\sum_{j=k_0}^{k+1} \Xi_{j-1} + \hat c \sum_{j =k_0}^{l(k_0)} \chi \big( \Rvalue_j - \psi(x^*) \big),
\end{align*}
where the first inequality results from the fact that each term
$ \Xi_j $ for $ j = l(k_0), \ldots, k+1 $ from the left-hand side
occurs $ m $ times within the double sum from the right-hand side
(observe that the relation $ l(j+1) = l(j)+1 $ holds for all $ j $),
whereas the remaining expressions $ \Xi_i $ are nonnegative,
the next inequality exploits \eqref{eq:lemmaxi} from 
\Cref{lemma:lemmaxi}, and the final estimate uses a telescoping 
sum argument where we omit some nonpositive summands. This completes our induction step.

Hence, it follows that $x^k \in B_\alpha(x^*)$ for all $k \geq k_0-1$.
Taking $k \to \infty$ in the resulting expression for $\sum_{j = k_0-1}^k
\| x^{j+1}- x^{j} \|$ from \eqref{eq:normxkp} then shows that 
$\{x^k\}_{k \in \mathbb{N}}$ is a Cauchy sequence and, therefore,
convergent. Thus, the accumulation point $x^*$ is the limit of the 
entire sequence $ \{ \xk \} $.
\end{proof}

\noindent
Next, we present a rate-of-convergence result for the case 
where the desingularization function is given by $\chi(t) = ct^{\kappa}$ for some $c > 0$ and $\kappa \in (0,1)$.
We recover the same rate-of-convergence as for the monotone proximal gradient method.
The proof is based on \cite{XiaoxiKanzowMehlitz2023} 
with suitable adaptations for the nonmonotone case. 

\begin{theorem}\label{thm:convergencerate}
Let \Cref{genas} hold, and suppose that $\{ \xk \}$ converges on some subsequence $\{ \xk\}_K$ to a limit point $x^*$ such that $\psi$ has the KL property at $x^*$. Then the entire sequence $\{ \xk \}$ converges to $x^*$. Further, if the corresponding desingularization function is given by $\chi(t) = ct^{\kappa}$ (for some $c>0$ and $\kappa \in (0,1)$), then the following statements hold:
\begin{enumerate}[label=(\alph*)]
	\item If $\kappa \in [1/2, 1)$, then $\{\Rvalue_k\}$ converges Q-linearly to $\psi(x^*)$ and $\{x^k\}$ converges R-linearly to $x^*$. \label{itm1}
	\item If $\kappa \in (0, 1/2)$, then there exist constants $\eta_1, \eta_2 > 0$ such that for all $k$ large enough it holds that\label{itm2}
	\begin{align}
		\Rvalue_k - \psi(x^*) &\leq \eta_1 k^{-\frac{1}{1-2\kappa}},\\
		\|\xk - x^*\| &\leq \eta_2 k^{-\frac{\kappa}{1-2\kappa}}.
	\end{align}
\end{enumerate}
\end{theorem}

\begin{proof}
Taking \Cref{thm:convergence} into account, we only need to verify the statements \ref{itm1} and \ref{itm2}. As a first step, let us prove the results for the $\{\Rvalue_k\}$. We first claim that for $\kappa \in (0,1)$ and with
\begin{equation*}
	\sigma := \frac{1-\alpha_{\max}}{2\gamma_{\max}} \frac{\pmin}{c^2\kappa^2 \Big(\frac{1}{\gammag} + L_\rho \Big)^2},
\end{equation*}
it holds that
\begin{equation}\label{eq:sigma-inequality}
	\psi(\xkp) - \psi(x^*) \leq \Big(\frac{1}{\sigma}\Big)^{\frac{1}{2(1-\kappa)}} \big( \Rvalue_k - \Rvalue_{k+1} \big)^{\frac{1}{2(1-\kappa)}}
\end{equation}
for all $k \in \mathbb{N}$ sufficiently large. In fact, if $\psi(\xkp) \leq \psi(x^*)$ holds, then the left-hand side of \eqref{eq:sigma-inequality}
is nonpositive, hence, the claim holds trivially,
Thus, it remains to consider the case $\psi(\xkp) > \psi(x^*)$. 
In view of our 
previous results, we may assume that $k$ is large enough such 
that $\xk \in B_\alpha (x^*)$ and $\psi(x^*) < \psi(x^{k+1}) < \psi(x^*) + \eta$ hold.

As $\psi$ satisfies the KL property at $x^*$ with $\chi(t) = ct^{\kappa}$, we have
\begin{align*}
	1 & \leq \chi' \big( \psi(\xkp) - \psi(x^*) \big)
	\dist \big( 0, \partial \psi(\xkp) \big)\\
	& = c\kappa \big( \psi(\xkp) - \psi(x^*) \big)^{\kappa-1}
    \dist \big( 0, \partial \psi(\xkp) \big).
\end{align*}
By \Cref{distbound}, this yields
\begin{equation*}
	1 \leq c\kappa \Big(\frac{1}{\gammag} + L_\rho\Big) 
	\big( \psi(\xkp) - \psi(x^*) \big)^{\kappa-1} \| \xkp - x^k \|,
\end{equation*}
which gives the inequality
\begin{equation}\label{eq:proofconvrate1}
	\|\xkp - x^k \| \geq \frac{1}{c \kappa \Big(
	\frac{1}{\gammag} + L_\rho \Big)} 
	\big( \psi(\xkp) - \psi(x^*) \big)^{1-\kappa}.
\end{equation}
By \eqref{Eq:Rupdate}, we also have 
\begin{equation}\label{eq:proofconvrate2}
	\Rvalue_{k+1} - \Rvalue_k \leq - \frac{1-\alpha_{\max}}{2\gamma_{\max}}\pmin \|\xkp - x^k \|^2 .
\end{equation}
Combination of \eqref{eq:proofconvrate1} and 
\eqref{eq:proofconvrate2} yields
\begin{align*}
	\Rvalue_{k+1} - \Rvalue_k &\leq - \frac{1-\alpha_{\max}}{2\gamma_{\max}} \pmin \|\xkp - x^k \|^2\\
	 & \leq - \frac{1-\alpha_{\max}}{2\gamma_{\max}} \pmin \frac{1}{c^2 \kappa^2 \Big(\frac{1}{\gammag} + L_\rho \Big)^2} 
	\big( \psi(\xkp) - \psi(x^*) \big)^{2(1-\kappa)}\\
	& = -\sigma \big( \psi(\xkp) - \psi(x^*) \big)^{2(1-\kappa)} .
\end{align*}
Rearranging these terms shows that the claim \eqref{eq:sigma-inequality}
holds.

Next recall that, by the acceptance criterion for the stepsize
$ \gamma_k $, we always have $ \psi (\xkp) \leq \Rvalue_k $.
Hence, it follows that
\begin{equation}\label{eq:Rvalue-inequality}
	\Rvalue_{k+1} = (1-p_{k+1}) \Rvalue_k + p_{k+1} \psi(\xkp) \leq (1-\pmin) \Rvalue_k + \pmin \psi(\xkp) .
\end{equation}
Denote by $\{s_k\}$ the sequence defined by $s_k := \Rvalue_k - \psi(x^*)
\geq 0 $. Then $ s_k \to 0 $ monotonically, and we obtain
\begin{align*}
	s_{k+1} &\leq (1-\pmin) s_k + \pmin(\psi(\xkp) - \psi(x^*))\\
			&\leq (1-\pmin)s_k + \pmin \Big(\frac{1}{\sigma}\Big)^{\frac{1}{2(1-\kappa)}} \big( s_k - s_{k+1} \big)^{\frac{1}{2(1-\kappa)}} ,
\end{align*}
where the first inequality follows from \eqref{eq:Rvalue-inequality}
and the second one results from \eqref{eq:sigma-inequality}.
This implies
\begin{align*}
s_k &\leq \frac{1}{\pmin} (s_k -s_{k+1}) + \Big(\frac{1}{\sigma}\Big)^{\frac{1}{2(1-\kappa)}} \big( s_k - s_{k+1} \big)^{\frac{1}{2(1-\kappa)}}\\
	&\leq \left(\frac{1}{\pmin} + \left(\frac{1}{\sigma}\right)^{\frac{1}{2(1-\kappa)}}\right) (s_k - s_{k+1})^{\min\{1, \frac{1}{2(1-\kappa)}\}}
\end{align*}
for all $ k $ sufficiently large.
As for all $a, b > 0$ it holds that $1/\min\{a, b\} = \max\{1/a, 1/b\}$, it follows that
\begin{equation*}
s_k^{\max\{1, 2(1-\kappa)\}} \leq \beta (s_k - s_{k+1}),
\end{equation*}
where 
\begin{equation*}
\beta := \left(\frac{1}{\pmin} + \left(\frac{1}{\sigma}\right)^{\frac{1}{2(1-\kappa)}}\right)^{\max\{1, 2(1-\kappa)\}} > 0
\end{equation*}
is a constant.

We are now in the setting of \Cref{Lem:Aragon} and immediately obtain the 
corresponding rate-of-convergence results for the sequence $\{\Rvalue_k\}$ as $\kappa \in (0, 1/2)$ implies that $2(1-\kappa) > 1$ and $\kappa \in [1/2, 1)$ implies that $2(1-\kappa) \in (0, 1]$.

Let us now verify the statements for the sequence $\{\xk\}$. In view of \Cref{thm:convergence}, the equations in the proof of that result remain valid if $k_0-1$ is replaced by some $k$ sufficiently large. Note that $\Xi_{j-1} \leq \sqrt{\Rvalue_{j-1} - \psi(x^*)} = \sqrt{s_{j-1}}$. Taking the monotonicity of the function $\chi$ and the sequences $\{\Rvalue_k\}$ and thus $\{s_k\}$ into account, it follows from \eqref{eq:normxkp} that, for $l>k$:
\begin{align*}
	\|\xk - x^l\| &\leq \sum_{j=k}^{l-1} \|x^{j+1}-x^j\| \\
	&\leq \frac{4}{\sqrt{a\pmin}} \sum_{j=k+1}^{l(k+1)} \Xi_{j-1} + \frac{2 \hat c}{\sqrt{a \pmin}} \sum_{j = k+1}^{l(k+1)}\chi \big( \Rvalue_j - \psi(x^*) \big)\\
	&\leq \frac{4}{\sqrt{a\pmin}} \sum_{j=k+1}^{l(k+1)} \sqrt{s_{j-1}} + \frac{2 \hat c}{\sqrt{a \pmin}} \sum_{j = k+1}^{l(k+1)}\chi ( s_j )\\
	&\leq \frac{4m}{\sqrt{a\pmin}} \sqrt{s_k} + \frac{2 \hat c m}{\sqrt{a \pmin}}\chi ( s_k )\\
	&\leq \tilde \eta s_k^{\min\{1/2, \kappa\}}
\end{align*}
for all $ k $ sufficiently large, where 
\begin{equation*}
\tilde \eta := \frac{4 + 2 \hat c c}{\sqrt{a\pmin}}m.
\end{equation*}
Taking now $l \to \infty$, together with the corresponding rate-of-convergence results for $\{s_k\}$ from the first part, this completes the proof.
\end{proof}


\noindent
We finally consider a generalized projected gradient method as a simple
application of our theory. 

\begin{example}\label{Ex:GenProjGrad}
\emph{(Generalized Nonmonotone Projected Gradient Method) \\
Consider the constrained optimization problem 
\begin{equation*}
	\min \ f(x) \quad \text{subject to} \quad x \in S
\end{equation*}
for some given set $ S \subseteq \X $ which is assumed to be nonempty
and closed (not necessarily convex). This problem can be reformulated
as the unconstrained composite optimization problem
\begin{equation*}
   \min \ f(x) + \phi (x), \qquad x \in \X,
\end{equation*}
where $ \phi (x) := \delta_S (x) $ denotes the indicator function
of $ S $, i.e., $ \delta_S (x) = 0 $ for $ x \in S $ and 
$ \delta_S (x) = + \infty $ otherwise. Since $ S $ is nonempty
and closed, this indicator function is proper and lower semicontinuous.
Moreover, the proximal subproblem \eqref{Eq:ProxIteration} reduces
to compute a projection of a given point onto the set $ S $. Recall
that this projection always exists, albeit it is not necessarily unique
since $ S $ is not assumed to be convex. Consequently, if $ f $
satisfies \Cref{genas} (c), and in particular if $ f $ further has the respective KL property,
our convergence results 
apply to this generalized projected gradient method. Note that
this may be viewed as a generalization of the famous spectral
gradient method from \cite{BirginMartinezRaydan2000} where the feasible 
set $ S $ is assumed to be convex (and the globalization uses the max-type
nonmonotone stepsize rule, see also the extension in
\cite{JiaKanzowMehlitzWachsmuth2021}).}
\end{example}



\section{Final Remarks}\label{Sec:Final}

The current paper presents a global and rate-of-convergence result
for nonmonotone proximal gradient methods applied to composite
optimization problems under fairly mild assumptions. To the best
of our knowledge, this is the first time that results of this kind
are shown for a nonmonotone method without a global Lipschitz assumption
or the a priori knowledge that the iterates generated by
the given method are bounded. Though the technique of proof is
quite technical and it is currently not clear whether these results
can be extended to other classes of first-order methods, we
plan to have a closer look at this topic as part of our future
research.\\[1mm]

\noindent 
{\bf Acknowledgement.} The authors thank Xiaoxi Jia for her comments 
on a previous version of the manuscript which helped to improve
\Cref{thm:convergencerate}.

\printbibliography

\end{document}